\documentclass[11pt,letterpaper]{amsart}
\usepackage{cases}
\usepackage{mathrsfs}
\usepackage{color}
\usepackage{latexsym,amssymb}
\usepackage{a4wide}
\usepackage{amsthm}
\usepackage{amsmath}
\usepackage{graphicx}
\input xy
\xyoption{all}
\usepackage{verbatim}
\usepackage[lite,abbrev,msc-links]{amsrefs}
\usepackage{framed}
\usepackage{bm}

\numberwithin{equation}{section}
\begin{document}
	\theoremstyle{plain}
	\newtheorem{thm}{Theorem}[section]
	\newtheorem{lem}[thm]{Lemma}
	\newtheorem{cor}[thm]{Corollary}
	\newtheorem{cor*}[thm]{Corollary*}
	\newtheorem{prop}[thm]{Proposition}
	\newtheorem{prop*}[thm]{Proposition*}
	\newtheorem{conj}[thm]{Conjecture}
	\theoremstyle{definition}
	\newtheorem{construction}{Construction}
	\newtheorem{notations}[thm]{Notations}
	\newtheorem{question}[thm]{Question}
	\newtheorem{prob}[thm]{Problem}
	\newtheorem{rmk}[thm]{Remark}
	\newtheorem{remarks}[thm]{Remarks}
	\newtheorem{defn}[thm]{Definition}
	\newtheorem{claim}[thm]{Claim}
	\newtheorem{assumption}[thm]{Assumption}
	\newtheorem{assumptions}[thm]{Assumptions}
	\newtheorem{properties}[thm]{Properties}
	\newtheorem{exmp}[thm]{Example}
	\newtheorem{comments}[thm]{Comments}
	\newtheorem{blank}[thm]{}
	\newtheorem{observation}[thm]{Observation}
	\newtheorem{defn-thm}[thm]{Definition-Theorem}
	\newtheorem*{Setting}{Setting}

	\newcommand{\sA}{\mathcal{A}}	\newcommand{\sB}{\mathcal{B}}	\newcommand{\sC}{\mathscr{C}}	\newcommand{\sD}{\mathscr{D}}	\newcommand{\sE}{\mathscr{E}}	\newcommand{\sF}{\mathscr{F}}	\newcommand{\sG}{\mathscr{G}}	\newcommand{\sH}{\mathscr{H}}	\newcommand{\sI}{\mathscr{I}}	\newcommand{\sJ}{\mathscr{J}}	\newcommand{\sK}{\mathscr{K}}	\newcommand{\sL}{\mathscr{L}}	\newcommand{\sM}{\mathscr{M}}	\newcommand{\sN}{\mathscr{N}}	\newcommand{\sO}{\mathcal{O}}	\newcommand{\sP}{\mathscr{P}}
	\newcommand{\sQ}{\mathscr{Q}}	\newcommand{\sR}{\mathscr{R}}	\newcommand{\sS}{\mathcal{S}}	\newcommand{\sT}{\mathscr{T}}	\newcommand{\sU}{\mathscr{U}}	\newcommand{\sV}{\mathscr{V}}	\newcommand{\sW}{\mathscr{W}}	\newcommand{\sX}{\mathcal{X}}	\newcommand{\sY}{\mathcal{Y}}	\newcommand{\sZ}{\mathcal{Z}}	\newcommand{\bZ}{\mathbb{Z}}	\newcommand{\bN}{\mathbb{N}}	\newcommand{\bQ}{\mathbb{Q}}	\newcommand{\bC}{\mathbb{C}}
	\newcommand{\bR}{\mathbb{R}}	\newcommand{\bH}{\mathbb{H}}	\newcommand{\bD}{\mathbb{D}}	\newcommand{\bE}{\mathbb{E}}	\newcommand{\bP}{\mathbb{P}}
	\newcommand{\bV}{\mathbb{V}}	\newcommand{\cV}{\mathcal{V}}	\newcommand{\cF}{\mathcal{F}}	\newcommand{\bfM}{\mathbf{M}}	\newcommand{\bfN}{\mathbf{N}}	\newcommand{\bfX}{\mathbf{X}}	\newcommand{\bfY}{\mathbf{Y}}	\newcommand{\spec}{\textrm{Spec}}	\newcommand{\dbar}{\bar{\partial}}	\newcommand{\ddbar}{\partial\bar{\partial}}	\newcommand{\redref}{{\color{red}ref}}
	
	
	\title[Arakelov inequality for families of pairs.] {Arakelov inequality for families of pairs.}
	
	\author[Junchao Shentu]{Junchao Shentu}
	\email{stjc@ustc.edu.cn}
	\address{School of Mathematical Sciences,
		University of Science and Technology of China, Hefei, 230026, China}

	\begin{abstract}
		We establish an Arakelov-type inequality for a morphism $f \colon (X,\Delta) \to S$, where $(X,\Delta)$ is a simple normal crossing semi-log canonical pair and $S$ is a smooth projective variety. As a consequence, we derive a bound on the Iitaka volumes of algebraic fiber spaces whose geometric generic fiber admits a good minimal model.
	\end{abstract}
	
	\maketitle
	\section{Introduction}
	The Arakelov inequality, first established by Arakelov \cite{Arakelov1971}, is a fundamental inequality governing numerical invariants of families of curves and plays a crucial role in the proof of the geometric Shafarevich conjecture-namely, the finiteness of admissible families of curves.
	This paper establishes a higher-dimensional Arakelov-type inequality over the complex numbers. The main result is stated as follows.
	\begin{thm}\label{thm_Arakelov_ineq}
		Let $f \colon (X,\Delta) \to S$ be a morphism from a projective simple normal crossing semi-log canonical (slc) pair $(X,\Delta)$ to a smooth projective variety $S$ of dimension $d$, with relative dimension $n = \dim X - \dim S$. Let $D_f \subset S$ be a proper algebraic subset such that $f$ restricts to a simple normal crossing family over $S \setminus D_f$, and let $D \subset S$ denote the codimension-one part of $D_f$. Denote by $R_f \subset D$ the ramification divisor of $f \colon (X,\Delta) \to S$.
		
		Let $$W \subset \left(f_\ast \mathcal{O}_X(kK_{X/S} + \lfloor k\Delta \rfloor)\right)^{\otimes r}$$ be a coherent subsheaf of rank $l$, for integers $k,r \geq 1$. If $K_S + D$ is pseudo-effective, then for every movable curve class $\alpha \in N_1(S)$ and every integer $m > 0$, one has  
		\begin{align}\label{align_mainthm_ineq1}
			c_1(W) \cdot \alpha \leq \frac{klrn}{2} \, d^{mklrn - 1} \, (K_S + D) \cdot \alpha + \frac{2}{m} \, D \cdot \alpha + r \, \deg\lceil k R_f \rceil \cdot \alpha.
		\end{align}
		Moreover, the following refined estimates hold:  
		\begin{itemize}  
			\item If $\dim S = 1$, then  
			\begin{align}\label{align_mainthm_ineq2}  
				\deg W \leq \frac{k l r n}{2} \, \deg(K_S + D) + r \, \deg \lceil k R_f \rceil.  
			\end{align}  
			\item If $(S,D_f)$ is log smooth (so that $D = D_f$) and $K_S + D$ is ample, then  
			\begin{align*}
				c_1(W) \cdot (K_S + D)^{d-1} \leq \frac{k l r n}{d} \, (K_S + D)^d + r \, \lceil k R_f \rceil \cdot (K_S + D)^{d-1}.  
			\end{align*}  
		\end{itemize}
	\end{thm}
    The assumption that $(X,\Delta)$ is a simple normal crossing pair-rather than a log smooth pair-ensures the broad applicability of Theorem \ref{thm_Arakelov_ineq} to families arising at the boundary of moduli spaces of varieties (c.f. \cite{Kollar1988,Kollar2010,Kollar2023,Birkar2022}).  
    Note that the ramification divisor $R_f$, as defined in Definition \ref{defn_ramified_divisor}, vanishes identically when the morphism $f \colon (X,\Delta) \to S$ is strictly semistable in codimension one (Definition \ref{defn_semistable_cod1}) .  
    The inequality (\ref{align_mainthm_ineq1}) is not sharp. A refined version is established in Theorem \ref{thm_Arakelov_family} and Corollary \ref{cor_Arakelov_ineq}.  Moreover, (\ref{align_mainthm_ineq2}) is sharp: equality holds for certain Shimura families (see \cite{VZ2004}, \cite{MVZ2006}). 
    
    As a consequence, we obtain the following upper bound for the Iitaka volume (defined in (\ref{align_Iitaka_volume})), adapting the argument of Lu-Tan-Zuo \cite{LTZ2017}.   
    
    \begin{thm}  
    	Let $f \colon X \to S$ be a morphism from a smooth projective variety $X$ to a smooth projective curve $S$, whose geometric generic fiber $F$ has dimension $n$ and admits a good minimal model. Let $D \subset S$ be a reduced effective divisor such that $f$ restricts to a smooth family over $S \setminus D$, and let $R_f \subset D$ denote the ramification divisor of $f$. Let $L$ be a nef line bundle on $S$. If $K_S + D$ is nef, then  
    	\begin{align*}  
    		\mathrm{Ivol}\big(\omega_{X/S} \otimes f^\ast L\big)  
    		\leq \mathrm{Ivol}(\omega_F) \cdot \big(\kappa(\omega_F) + 1\big) \cdot \left( \frac{n}{2} \deg(K_S + D) + \deg L \right)  
    		+ \limsup_{k \to \infty} \frac{\big(\kappa(\omega_F) + 1\big)!}{k^{\kappa(\omega_F) + 1}} \deg \lceil k R_f \rceil.  
    	\end{align*}  
    \end{thm}      
    In particular, if $\kappa(\omega_F) > 1$ or $R_f = 0$ (i.e. when $f$ is semistable) then  
    \begin{equation*}  
    	\mathrm{Ivol}\big(\omega_{X/S} \otimes f^\ast L\big)  
    	\leq \mathrm{Ivol}(\omega_F) \cdot \big(\kappa(\omega_F) + 1\big) \cdot \left( \frac{n}{2} \deg(K_S + D) + \deg L \right).  
    \end{equation*}  
    
    A substantial body of work has addressed the Arakelov inequality over one-dimensional bases, including \cite{Deligne1987,Kovacs1996,Kovacs1997,Kovacs2000,Peters2000,Viehweg2000,Zuo2017}. For a comprehensive historical overview of this development, we refer the reader to \cite{Viehweg2009}.  
    
    When $\dim S = 1$, $\Delta = 0$, $r = 1$, and $X$ is smooth, inequality (\ref{align_mainthm_ineq2}) gives an effective form of the Arakelov-type inequality established by Viehweg-Zuo \cite{VZ2001}. In the same setting-with the additional assumption that $f$ is strictly semistable-the inequality was independently obtained by Viehweg-Zuo \cite{VZ2006} and M\"oller-Viehweg-Zuo \cite{MVZ2006}. Furthermore, Lu-Yang-Zuo \cite{LYZ2022} showed that (\ref{align_mainthm_ineq2}) is strict for all integers $k$ such that the $k$-th pluricanonical linear system of the general fiber induces a birational map, provided $r = 1$, $\Delta = 0$, and $f$ is a non-isotrivial semistable family of projective manifolds of general type over a curve. Finally, Theorem \ref{thm_Arakelov_ineq} provides an effective version of an earlier work of Kov\'acs-Taji \cite{Kovacs2024}.

    This paper builds upon the results of \cite{stjc2026}, in which an Arakelov inequality was established to confirm the geometric Shafarevich boundedness conjecture for both the moduli space of Koll\'ar-Shepherd-Barron (KSB) pairs and that of stable minimal models. The central technical foundation of the present work is the Hodge theoretic Arakelov inequality established in \cite[Theorems 2.10 and 2.13]{stjc2026}.
    
    \subsection{Comparing to recent work of Kov\'acs-Taji \cite{Kovacs2024}}
    In \cite{Kovacs2024}, Kov\'acs and Taji established a higher dimensional Arakelove inequality when the fibers are canonically polarized manifolds. Theorem \ref{thm_Arakelov_ineq} provides an effective version of Kov\'acs-Taji's inequality. To illustrate the difference between the techniques of the present paper and \cite{Kovacs2024}, assume for simplicity that $f^o:X^o\to S^o$ is a family of canonically polarized manifolds admitting an extension to a family $f:X\to S$ of KSB-stable varieties. In \cite{Kovacs2024}, the authors construct an embedding  
    \begin{align}\label{align_intro_embed}
    	\xi^\ast\lambda_{0,r}^{\otimes r}\simeq \det f_\ast\big(\mathcal{O}_X(rK_{X/S})\big)^{\otimes r} \hookrightarrow f^{[\mu]}_{\ast}\big(\mathcal{O}_{X^{[\mu]}}(rK_{X^{[\mu]}/S})\big),
    \end{align}  
    valid for sufficiently large $\mu$, where $\mu$ depends only on the some discrete invariants. Here, $f^{[\mu]}:X^{[\mu]}\to S$ denotes the $\mu$-fold fiber product of $f$ over $S$. The embedding gives rise to an embedding of $\det f_\ast\big(\mathcal{O}_X(rK_{X/S})\big)^{\otimes r}$ into a variation of Hodge structures. Then the authors apply the positivity of the logarithmic cotangent bundle $\Omega_S(\log D)$ (\cite{CP2019}) to derive a higher-dimensional Arakelov-type inequality.
    
    In contrast to \cite{Kovacs2024}, the central technical innovation of this paper consists in replacing the embedding (\ref{align_intro_embed}) with one arising from the alternating sum construction:  
    \[
    \det f_\ast\big(\mathcal{O}_X(rK_{X/S})\big)^{\otimes r} \hookrightarrow \bigotimes^{lr} f_\ast\big(\mathcal{O}_X(rK_{X/S})\big) \simeq f^{[lr]}_{\ast}\big(\mathcal{O}_{X^{[lr]}}(rK_{X/S})\big),
    \]  
    where $l = \operatorname{rank} f_\ast\big(\mathcal{O}_X(rK_{X/S})\big)$. By the geometric characterization of the canonical extension of admissible variations of mixed Hodge structure of geometric origin (\cite{stjc2025}), the embeding give rise to an embedding of $\det f_\ast\big(\mathcal{O}_X(rK_{X/S})\big)^{\otimes r}$ into the canonical extension of a variation of Hodge structure on $S^o$. Since $lr<\mu$ in general, this construction leads to a quantitatively sharper Arakelov-type inequality. Furthermore, in addition to the positivity of $\Omega_S(\log D)$, we exploit the parabolic semistability of the Deligne extension of the variation of Hodge structure on $S^o$ (\cite{Simpson1990}) to establish the optimal Arakelov inequality (\ref{align_mainthm_ineq2}).
    
	\textbf{Notation and conventions.}
	\begin{itemize}
		\item Let $F$ be a torsion-free coherent sheaf on a smooth algebraic variety $X$. Define its dual by $F^\vee := \mathcal{H}om_{\mathcal{O}_X}(F, \mathcal{O}_X)$; then $F^{\vee\vee}$ is the reflexive hull of $F$. The determinant line bundle $\det(F)$ is defined as the reflexive hull of $\bigwedge^{\operatorname{rank} F} F$.
	\end{itemize}


	\section{Arakelov Inequality}
	\subsection{Ramified divisor}\label{section_ramified_divisor}
	\begin{defn}[Simple normal crossing family]\label{defn_SNC_family}
		Let $S$ and $X$ be reduced schemes of finite type over $\operatorname{Spec}(\mathbb{C})$, and let $D \geq 0$ be a Weil $\bQ$-divisor on $X$. Let $f: X \to S$ be a morphism. The morphism $f: (X, D) \to S$ is said to be \emph{simple normal crossing over $S$} at a point $x \in X$ if there exists a Zariski open neighborhood $U$ of $x$ in $X$ that can be embedded into a scheme $Y$, which is smooth over $S$. In this embedding, $Y$ admits a regular system of parameters $(z_1, \dots, z_p, y_1, \dots, y_r)$ over $S$ at the point corresponding to $x = 0$, such that $U$ is defined by the monomial equation $z_1 \cdots z_p = 0$ and  
		\[ D|_U = \sum_{i=1}^r a_i (y_i = 0)|_U, \quad \text{where } a_i \geq 0, \]  
		over $S$.
		
		The morphism $f: (X, D) \to S$ is said to be a \emph{simple normal crossing family over $S$} if it satisfies the simple normal crossing condition over $S$ at every point of $X$. 
		
		In the special case where $S = \operatorname{Spec}(\mathbb{C})$, the pair $(X, D)$ is referred to as a \emph{simple normal crossing pair} if $(X, D)$ forms a simple normal crossing family over $\operatorname{Spec}(\mathbb{C})$. $(X, D)$ is called a \emph{simple normal crossing slc pair} if the coefficients of $\Delta$ lie in $[0,1]$. $X$ is referred to as a \emph{simple normal crossing scheme} if $(X,0)$ is a simple normal crossing pair. In this context, $X$ has Gorenstein singularities and possesses an invertible dualizing sheaf $\omega_X$. The canonical divisor $K_X$ is defined up to linear equivalence via the isomorphism $\omega_X \simeq \sO_X(K_X)$.
	\end{defn}
	\begin{defn}[Semistable morphism in codimension one]\label{defn_semistable_cod1}  
		Let $f \colon (X,\Delta) \to S$ be a morphism from a simple normal crossing scheme $X$ to a smooth variety $S$, where $\Delta \geq 0$ is an effective $\mathbb{Q}$-divisor on $X$. The morphism $f$ is called \emph{semistable} if there exists a (not necessarily connected) smooth divisor $D_f \subset S$ satisfying the following conditions:  
		\begin{enumerate}  
			\item $(X,\Delta)$ is a simple normal crossing pair.  
			\item $f \colon (X,\Delta) \to S$ is a simple normal crossing family over $S \setminus D_f$.  
			\item $f^{-1}(D_f)$ and $f^{-1}(D_f) \cup \Delta$ are simple normal crossing divisors on $X$.  
			\item For every stratum $Z$ of $(f^{-1}(D_f))_{\mathrm{red}} \cup \Delta_{\mathrm{red}}$, the restriction $f|_Z \colon Z \to D_f$ is surjective onto an irreducible component of $D_f$.  
		\end{enumerate}  
		If, in addition, $\Delta$ is horizontal (i.e., contains no component of any fiber) and $f^{-1}(D_f)$ is reduced, then $f \colon (X,\Delta) \to S$ is said to be \emph{strictly semistable}.  
		
		The morphism $f \colon (X,\Delta) \to S$ is called \emph{semistable in codimension one} (resp. \emph{strictly semistable in codimension one}) if there exists a dense Zariski open subset $U \subset S$ with $\mathrm{codim}_S(S \setminus U) \geq 2$ such that the base change $f|_{f^{-1}(U)} \colon \big(f^{-1}(U),\, \Delta|_{f^{-1}(U)}\big) \to U$ is semistable (resp. strictly semistable).  
	\end{defn}  
    Given a proper morphism $f \colon (X,\Delta) \to S$ from a simple normal crossing pair $(X,\Delta)$ to a smooth variety $S$, there exist closed algebraic subsets $Z_X \subset X$ and $Z_S \subset S$ satisfying the following conditions:  
    \begin{itemize}  
    	\item $\mathrm{codim}_X(Z_X) \geq 2$ and $\mathrm{codim}_S(Z_S) \geq 2$;  
    	\item Setting $U := X \setminus \big(f^{-1}(Z_S) \cup Z_X\big)$ and $V := S \setminus Z_S$, the induced morphism $f|_U \colon \big(U,\, \Delta|_U\big) \to V$ is semistable.  
    \end{itemize}      
    Moreover, $\Delta$ admits a canonical decomposition  
    \begin{align}\label{align_decomp_div}  
    	\Delta = \Delta^h + \Delta^v + \Delta^e,  
    \end{align}  
    where  
    \begin{itemize}  
    	\item $\Delta^h$ is the \emph{horizontal part} of $\Delta$: the sum of all components of $\Delta$ that dominate $S$;  
    	\item $\Delta^v$ is the \emph{vertical part} of $\Delta$: the sum of all components of $\Delta$ whose image in $S$ has codimension one;  
    	\item $\Delta^e$ is the \emph{exceptional part} of $\Delta$: the sum of all components of $\Delta$ whose image in $S$ has codimension at least two.  
    \end{itemize}  
    Let $D_f \subset S$ be a closed algebraic subset such that $f$ restricts to a simple normal crossing family over $S \setminus D_f$, and let $D \subset S$ denote the codimension-one part of $D_f$.      
    \begin{defn}\label{defn_ramified_divisor}  
    	The \emph{ramification divisor} associated with $f \colon (X,\Delta) \to S$ is defined as the smallest effective $\mathbb{Q}$-divisor $R_f$ on $S$ satisfying  
    	\begin{align}\label{align_ramified_div}  
    		f^{-1} R_f \geq f^\ast D - (f^\ast D)_{\mathrm{red}} + \Delta^v  
    	\end{align}  
    	over $U$ described as above. Here, $f^{-1} R_f$ denotes the closure in $X$ of the pullback $(f|_U)^\ast R_f$.  
    \end{defn}    
    The ramification divisor $R_f$ depends only on the family $f \colon (X,\Delta) \to S$ and is independent of the choice of $D_f$, $U$, or $V$. When the coefficients of $\Delta^v$ lie in $[0,1]$, one has $0 \leq R_f \leq D$. Moreover, if $f$ is strictly semistable in codimension one, then $R_f = 0$.  
    
	By applying semi-log canonical desingularization to the total space (c.f. \cite{Bierstone2013}), any family $f \colon (X,\Delta) \to S$ can be modified so that it becomes semistable in codimension one.  
	Moreover, using Kawamata’s covering construction \cite[Theorem 17]{Kawamata1981}, one may further alter the base and total space-via a finite surjective morphism $S' \to S$ and the associated pullback family-to achieve strict semistability in codimension one.  
	\begin{prop}[Strictly semistable reduction in codimension one]\label{prop_ssr_cod1}  
		Let $f \colon X \to S$ be a proper surjective morphism from a simple normal crossing scheme $X$ to a smooth variety $S$, and let $\Delta \geq 0$ be a $\mathbb{Q}$-divisor on $X$. Assume that $f \colon (X,\Delta) \to S$ is semistable in codimension one, and that there exists a simple normal crossing divisor $D_f \subset S$ such that $f$ restricts to a simple normal crossing family over $S \setminus D_f$. Then there exists a commutative diagram  
		\begin{align}\label{align_semistablereduction_cod1}  
			\xymatrix{  
				\widetilde{X} \ar@/^/[rr]^{\widetilde{\tau}} \ar[dr]_{\widetilde{f}} \ar[r]^-{\tau} & X \times_S \widetilde{S} \ar[r] \ar[d] & X \ar[d]^f \\  
				& \widetilde{S} \ar[r]^{\sigma} & S  
			}  
		\end{align}  
		satisfying the following conditions:  
		\begin{enumerate}  
			\item $\widetilde{S}$ is a smooth projective variety, $\sigma$ is a flat finite surjective morphism, and $\sigma^{-1}(D_f)$ is a reduced simple normal crossing divisor on $\widetilde{S}$;  
			\item $\tau$ is a functorial desingularization of the main component of $X \times_S \widetilde{S}$-in particular, $\tau$ is an isomorphism over the open subset $\big(\sigma^{-1}(S \setminus D_f)\big) \times_{S \setminus D_f} f^{-1}(S \setminus D_f)$;  
			\item Let $\widetilde{\Delta}$ denote the strict transform of the horizontal part of $\Delta$ under $\widetilde{\tau} \colon \widetilde{X} \to X$, i.e., $\widetilde{\tau}^\ast \Delta = \widetilde{\Delta} + E$, where $E \geq 0$ is a $\mathbb{Q}$-divisor supported entirely on $\widetilde{f}^{-1}\big(\sigma^{-1}(D_f)\big)$ and sharing no common components with $\widetilde{\Delta}$. Then $\widetilde{f} \colon (\widetilde{X}, \widetilde{\Delta}) \to \widetilde{S}$ is strictly semistable in codimension one.  
		\end{enumerate}  
	\end{prop}  
	
	\begin{lem}\label{lem_functorial_pushforward_pluricanonical}  
		Let the notation be as in Proposition \ref{prop_ssr_cod1}. Suppose further that $f \colon (X,\Delta) \to S$ is semistable and $\widetilde{f} \colon (\widetilde{X},\widetilde{\Delta}) \to \widetilde{S}$ is strictly semistable. Let $R_f$ denote the ramification divisor associated with $f$, and let $k > 0$ be an integer. Then pullback of pluricanonical forms induces an injective morphism  
		\begin{equation*}  
			\sigma^\ast f_\ast \sO_X\big(kK_{X/S} + \lfloor k\Delta - k f^\ast R_f \rfloor\big)  
			\hookrightarrow \widetilde{f}_\ast \sO_{\widetilde{X}}\big(kK_{\widetilde{X}/\widetilde{S}} + k\widetilde{\Delta}\big).  
		\end{equation*}  
	\end{lem}  
	\begin{proof}
		Denote $S^o := S \setminus D_f$, $\widetilde{S}^o := \sigma^{-1}(S^o)$, and $D_{\widetilde{f}} := \widetilde{S} \setminus \widetilde{S}^o$. Observe that the family  
		\[
		\widetilde{f}^o \colon \big(\widetilde{X}^o := \widetilde{f}^{-1}(\widetilde{S}^o),\, \widetilde{\Delta}|_{\widetilde{X}^o}\big) \to \widetilde{S}^o  
		\]  
		is the base change of the simple normal crossing family  
		\[
		f^o \colon \big(X^o := f^{-1}(S^o),\, \Delta|_{X^o}\big) \to S^o.  
		\]  
		Consequently, there is a natural injection
		\begin{align}\label{align_pullback_pushforward2}  
			f^o_\ast \sO_{X^o}\big(kK_{X^o/S^o} + k\Delta|_{X^o}\big)  
			\to
			\sigma_\ast \widetilde{f}^o_\ast \sO_{\widetilde{X}^o}\big(kK_{\widetilde{X}^o/\widetilde{S}^o} + k\widetilde{\Delta}|_{\widetilde{X}^o}\big).  
		\end{align}  		
		To prove the lemma, it suffices to show that (\ref{align_pullback_pushforward2}) extends uniquely to an injective morphism  
		\begin{align}\label{align_pullback_pushforward3}  
			f_\ast \sO_X\big(kK_{X/S} + \lfloor k\Delta - k f^\ast R_f \rfloor\big)  
			\hookrightarrow  
			\sigma_\ast \widetilde{f}_\ast \sO_{\widetilde{X}}\big(kK_{\widetilde{X}/\widetilde{S}} + k\widetilde{\Delta}\big).  
		\end{align}  		
		Let $E := f^{-1}(D_f)$ and $\widetilde{E} := \widetilde{f}^{-1}(D_{\widetilde{f}})$ denote the schematic preimages, since $f$ and $\widetilde{f}$ are smooth over $S^o$ and $\widetilde{S}^o$, respectively, and $D_f$, $D_{\widetilde{f}}$ are snc. The pullback of relative logarithmic canonical forms yields a nonzero morphism  
		\begin{align*}  
			\widetilde{\tau}^\ast \sO_X\big(K_{X/S} + E_{\mathrm{red}} - E\big)  
			\longrightarrow  
			\sO_{\widetilde{X}}\big(K_{\widetilde{X}/\widetilde{S}}\big).  
		\end{align*}  
		
		Decompose $\Delta = \Delta^h + \Delta^v$ as in (\ref{align_decomp_div}), where $\Delta^h$ is the horizontal part of $\Delta$ and $\mathrm{supp}(\Delta^v) \subset E$. Since $\tau$ is birational and an isomorphism over $\widetilde{S}^o$, and $\Delta^h$ is horizontal, the strict transform satisfies $\widetilde{\Delta} = \widetilde{\tau}^\ast \Delta^h$.  
		
		By the definition of the ramification divisor (\ref{align_ramified_div}), we have  
		\[
		f^\ast R_f \geq E - E_{\mathrm{red}} + \Delta^v,  
		\]  
		and hence  
		\[
		\lfloor k\Delta^v - k f^\ast R_f \rfloor \leq \lfloor k\Delta^v - k(E - E_{\mathrm{red}} + \Delta^v) \rfloor \leq k(E_{\mathrm{red}} - E).
		\]  
		It follows that the natural inclusion  
		\[
		\sO_X\big(kK_{X/S} + \lfloor k\Delta^v - k f^\ast R_f \rfloor\big)  
		\hookrightarrow  
		\sO_X\big(kK_{X/S} + k(E_{\mathrm{red}} - E)\big)  
		\]  
		composes with the pullback map above to give an injective morphism  
		\[
		\widetilde{\tau}^\ast \sO_X\big(kK_{X/S} + \lfloor k\Delta^v - k f^\ast R_f \rfloor\big)  
		\hookrightarrow  
		\sO_{\widetilde{X}}\big(kK_{\widetilde{X}/\widetilde{S}}\big).  
		\]  
		
		Tensoring both sides with $\widetilde{\tau}^\ast \sO_X(k\Delta^h) \simeq \sO_{\widetilde{X}}(k\widetilde{\Delta})$, we obtain an injective morphism  
		\[
		\widetilde{\tau}^\ast \sO_X\big(kK_{X/S} + \lfloor k\Delta - k f^\ast R_f \rfloor\big)  
		\hookrightarrow  
		\sO_{\widetilde{X}}\big(kK_{\widetilde{X}/\widetilde{S}} + k\widetilde{\Delta}\big).  
		\]  
		
		Applying the pushforward $\widetilde{\tau}_\ast$ (which preserves injectivity, as $\tau$ is proper and birational), we deduce  
		\[
		\sO_X\big(kK_{X/S} + \lfloor k\Delta - k f^\ast R_f \rfloor\big)  
		\hookrightarrow  
		\widetilde{\tau}_\ast \sO_{\widetilde{X}}\big(kK_{\widetilde{X}/\widetilde{S}} + k\widetilde{\Delta}\big).  
		\]  
		
		Finally, applying $f_\ast$ and using the commutativity $f_\ast \widetilde{\tau}_\ast = \sigma_\ast \widetilde{f}_\ast$ (by the diagram (\ref{align_semistablereduction_cod1})), we obtain the desired injective morphism  
		\begin{align*}  
			f_\ast \sO_X\big(kK_{X/S} + \lfloor k\Delta - k f^\ast R_f \rfloor\big)  
			\hookrightarrow  
			\sigma_\ast \widetilde{f}_\ast \sO_{\widetilde{X}}\big(kK_{\widetilde{X}/\widetilde{S}} + k\widetilde{\Delta}\big),  
		\end{align*}  
		thereby establishing (\ref{align_pullback_pushforward3}).  
	\end{proof}
    We will also require the following lemma, originally due to Viehweg \cite{Viehweg1983}.
    \begin{lem}(\cite[Lemma 5.5]{stjc2025})\label{lem_Viehweg_inclusion}
    	Let $g: Y \to S$ be a proper surjective morphism from a simple normal crossing variety $Y$ to a smooth variety $S$. Let $\tau: S' \to S$ be a flat projective surjective morphism from a smooth variety $S'$. Let $\Delta \geq 0$ be an integral Cartier divisor on $Y$. Consider the following commutative diagram:  
    	\begin{align}  
    		\xymatrix{  
    			Y \ar[d]^g & Y'' \ar[l]_{\rho'} \ar[d]^{g''} & Y' \ar[l]_{\rho} \ar[dl]^{g'} \\  
    			S & S' \ar[l]_{\tau} &  
    		},  
    	\end{align}  
    	where $g'': Y'' \to S'$ is the base change of $g$, and $\rho: Y' \to Y''$ is a semi-log resolution of singularities. Then, for every $r \geq 1$, there exists a natural inclusion  
    	$$  
    	g'_\ast\left(\sO_{Y'}(rK_{Y'/S'} + (\rho'\rho)^\ast\Delta)\right) \subset \tau^\ast g_\ast\left(\sO_Y(rK_{Y/S} + \Delta)\right).  
    	$$  
    \end{lem}
	\subsection{Fiber product}\label{section_fiberprod}
	Let $f \colon (X, \Delta) \to S$ be a morphism from a simple normal crossing pair $(X, \Delta)$ to a smooth variety $S$. Its $r$-fold fiber product $f^{[r]} \colon (X^{[r]}_S, \Delta^{[r]}_S) \to S$ is defined as follows:  
	\begin{itemize}  
		\item $X^{[r]}_S$ denotes the $r$-fold fiber product $X \times_S X \times_S \cdots \times_S X$;  
		\item $\Delta^{[r]}_S := \sum_{i=1}^r p_i^\ast \Delta$, where $p_i \colon X^{[r]}_S \to X$ is the projection onto the $i$-th factor;  
		\item $f^{[r]} \colon (X^{[r]}_S, \Delta^{[r]}_S) \to S$ is the natural structure morphism.  
	\end{itemize}  
	
	Assume that $f \colon (X,\Delta) \to S$ is strictly semistable. Then $f$ is a locally stable family of slc pairs in the sense of Koll\'ar \cite{Kollar2023}. Under this assumption, $f^{[r]} \colon (X^{[r]}_S, \Delta^{[r]}_S) \to S$ is a locally stable family of slc pairs (\cite[Corollary 4.3]{WeiWu2023}), and $(X^{[r]}_S, \Delta^{[r]}_S)$ itself is an slc pair (\cite[Theorem 4.54]{Kollar2023}).  
	
	Let  
	\[
	\tau \colon (X^{(r)}_S, \Delta^{(r)}_S) \to (X^{[r]}_S, \Delta^{[r]}_S)  
	\]  
	be a semi-log canonical resolution (\cite{Bierstone2013} or \cite[\S10.4]{Kollar2013}), where $\Delta^{(r)}_S$ is a reduced simple normal crossing divisor satisfying  
	\begin{align*}
		\tau^\ast(K_{X^{[r]}_S} + \Delta^{[r]}_S) = K_{X^{(r)}_S} + \Delta^{(r)}_S - E,  
	\end{align*}  
	with $E$ an effective $\tau$-exceptional $\mathbb{Q}$-divisor whose support contains no component of $\Delta^{(r)}_S$.  	
	\begin{lem}[\cite{stjc2025}, Lemma 5.4]\label{lem_mild_pushforward}  
		Assume that $f \colon (X, \Delta) \to S$ is strictly semistable, and let $k \geq 1$ be an integer such that $kK_{X/S} + k\Delta$ is Cartier. Then:  
		\begin{enumerate}  
			\item $\tau_\ast \sO_{X^{(r)}_S}\big(kK_{X^{(r)}_S/S} + k\Delta^{(r)}_S\big) \simeq \sO_{X^{[r]}_S}\big(kK_{X^{[r]}_S/S} + k\Delta^{[r]}_S\big)$;  
			\item $f^{[r]}_\ast \sO_{X^{[r]}_S}\big(kK_{X^{[r]}_S/S} + k\Delta^{[r]}_S\big)$ is reflexive;  
			\item $f^{[r]}_\ast \sO_{X^{[r]}_S}\big(kK_{X^{[r]}_S/S} + k\Delta^{[r]}_S\big) \simeq \big(f_\ast \sO_X(kK_{X/S} + k\Delta)\big)^{\otimes r \vee\vee}$.  
		\end{enumerate}  
	\end{lem} 
	\subsection{Proof of Theorem \ref{thm_Arakelov_ineq}}\label{section_VZHiggs_stable_family}
	\subsubsection{}
	We fix a projective morphism $f \colon X \to S$ from a projective simple normal crossing slc pair $(X,\Delta)$ to a smooth projective variety $S$. Let $d = \dim S$ and $n = \dim X - \dim S$. Suppose there exists a reduced simple normal crossing divisor $D_f \subset S$ satisfying:  
	\begin{enumerate}  
		\item[(i)] the restriction $f^o := f|_{X^o} \colon (X^o, \Delta^o) \to S^o$ is a simple normal crossing family, where $S^o := S \setminus D_f$, $X^o := f^{-1}(S^o)$, and $\Delta^o := \Delta|_{X^o}$;  
		\item[(ii)] no irreducible component of $f^{-1}(D_f)$ is contained in $\mathrm{supp}(\Delta)$.  
	\end{enumerate}  
	
	Let $L$ be a torsion-free coherent sheaf of rank one on $S$, and suppose there exists a nonzero morphism  
	\begin{align}\label{align_s}  
		s_L \colon L^{\otimes k} \longrightarrow f_\ast\big(\sO_X(kK_{X/S} + k\Delta)\big)  
	\end{align}  
	for some integer $k \geq 1$ such that $k\Delta$ is Cartier.  
	
	\begin{thm}[Theorem 2.13 \cite{stjc2026}]\label{thm_VZ_Arakelov}  
		Let the notation be as above. If $K_S + D_f$ is pseudo-effective, then  
		\begin{align}\label{align_Arakelov_ineq_movable}  
			\big(c_1(L) - D_f\big) \cdot \alpha \leq \frac{1}{n+1} \sum_{p=1}^{n} p \, d^{p-1} \, \big(K_S + D_f\big) \cdot \alpha  
		\end{align}  
		for every movable curve class $\alpha \in N_1(S)$.  
		
		If, moreover, $K_S + D_f$ is ample, then the absolute Arakelov-type inequality  
		\begin{align}\label{align_Araineq_abs1}  
			\big(c_1(L) - [D_f]\big) \cdot \big(K_S + D_f\big)^{d-1} \leq \frac{n}{d} \, \big(K_S + D_f\big)^d  
		\end{align}  
		holds.  
	\end{thm} 
\subsubsection{}We are now ready to prove Theorem \ref{thm_Arakelov_ineq}.  

\begin{thm}\label{thm_Arakelov_family}  
	Let $f \colon (X,\Delta) \to S$ be a projective morphism from a projective simple normal crossing slc pair $(X,\Delta)$ to a smooth projective variety $S$ of dimension $d$. Let $d = \dim S$ and $n = \dim X - \dim S$. Let $D_f \subset S$ be a closed algebraic subset such that $f$ restricts to a simple normal crossing family over $S \setminus D_f$, and let $D \subset S$ denote the codimension-one part of $D_f$. Let $R_f\leq D$ be the ramification divisor associated with $f \colon (X,\Delta) \to S$.  
	
	Let $k, r \geq 1$ be integers such that $k\Delta$ is integral, and let  
	\[
	W \subset \Big(f_\ast \mathscr{O}_X\big(kK_{X/S} + \lfloor k\Delta - k f^\ast R_f \rfloor\big)^{\otimes r}\Big)^{\vee\vee}  
	\]  
	be a coherent subsheaf of rank $l$. If $K_S + D$ is pseudo-effective, then for every integer $m > 0$ and every movable curve class $\alpha \in N_1(S)$,  
	\begin{align*}
		c_1(W)\cdot \alpha &\leq \frac{1}{m(mklrn + 1)}\sum_{p=1}^{mklrn}pd^{p-1}(K_S + D) \cdot \alpha+\frac{2}{m}D\cdot \alpha
	\end{align*}
	
	If, moreover, $(S,D)$ is log smooth and $K_S + D$ is ample, then the inequality  
	\begin{align*}
		c_1(W)\cdot(K_S+D)^{d-1} \leq \frac{klrn}{d}(K_S+D)^{d}
	\end{align*} 
	holds.  
\end{thm}  
	\begin{proof}		
		\emph{Step 1: Strictly semistable reduction in codimension one.}  
		After possibly performing blowups with smooth centers on $S$, we may assume that $D_f$ is a reduced simple normal crossing divisor ($D_f = D$ in this case). Subsequently, by a finite sequence of smooth blowups on $X$, we may further arrange that $f \colon (X,\Delta) \to S$ is semistable in codimension one.  
		
		Let  
		\begin{align*}  
			\xymatrix{  
				(\widetilde{X},\widetilde{\Delta}) \ar[d]_{\widetilde{f}} \ar[r]^{\tau} & (X,\Delta) \ar[d]^{f} \\  
				\widetilde{S} \ar[r]^{\sigma} & S  
			}  
		\end{align*}  
		be a strictly semistable reduction of $f$ in codimension one, as provided by Proposition \ref{prop_ssr_cod1}.  
		
		Let $Z \subset S$ be a closed algebraic subset of codimension at least two such that:  
		\begin{itemize}  
			\item $\sigma$ is finite and flat over $S \setminus Z$;  
			\item the restriction $f_1 := f|_{X_1} \colon (X_1 := f^{-1}(S \setminus Z),\, \Delta_1 := \Delta|_{X_1}) \to S_1 := S \setminus Z$ is semistable;  
			\item the restriction $\widetilde{f}_1 := \widetilde{f}|_{\widetilde{X}_1} \colon (\widetilde{X}_1 := \widetilde{f}^{-1}(\widetilde{S} \setminus \sigma^{-1}(Z)),\, \widetilde{\Delta}_1 := \widetilde{\Delta}|_{\widetilde{X}_1}) \to \widetilde{S}_1 := \widetilde{S} \setminus \sigma^{-1}(Z)$ is strictly semistable.  
		\end{itemize}  
		
		Let $S^\circ := S \setminus D_f$, $X^\circ := f^{-1}(S^\circ)$, $f^\circ := f|_{X^\circ}$, and $l := \mathrm{rank}(W)$. Let $m > 0$ be an integer.  		
		Using the notation introduced in \S\ref{section_fiberprod}, denote by $(\widetilde{X}^{[mklr]}_{1,\widetilde{S}_1}, \widetilde{\Delta}^{[mklr]}_{1,\widetilde{S}_1})$ the $mklr$-fold fiber product of $(\widetilde{X}_1,\widetilde{\Delta}_1)$ over $\widetilde{S}_1$. Let  
		\[
		(\widetilde{X}^{(mklr)}_1, \widetilde{\Delta}^{(mklr)}_1) \to (\widetilde{X}^{[mklr]}_{1,\widetilde{S}_1}, \widetilde{\Delta}^{[mklr]}_{1,\widetilde{S}_1})  
		\]  
		be a functorial semi-log canonical desingularization biholomorphic over the smooth locus $\widetilde{X}^{o,[mklr]}_{\widetilde{S}^\circ}$.  
		
		Denote the induced structure morphism by  
		\[
		f^{(mklr)}_1 \colon (\widetilde{X}^{(mklr)}_1, \widetilde{\Delta}^{(mklr)}_1) \to \widetilde{S}_1.  
		\]   
		
		Let $(X^{[mklr]}_{1,S_1}, \Delta^{[mklr]}_{1,S_1})$ denote the $mklr$-fold fiber product of $(X_1,\Delta_1)$ over $S_1$, where  
		\[
		\Delta^{[mklr]}_1 := \sum_{i=1}^{mklr} p_i^\ast \Delta_1,  
		\]  
		and $p_i \colon X^{[mklr]}_{1,S_1} \to X_1$ is the projection onto the $i$-th factor.  
		
		Let  
		\[
		\widetilde{\pi} \colon (X^{(mklr)}_1, \Delta^{(mklr)}_1) \to (X^{[mklr]}_{1,S_1}, \Delta^{[mklr]}_{1,S_1})  
		\]  
		be a functorial semi-log canonical desingularization biholomorphic over $X^{o,[mklr]}_{S^o}$. Here, $\Delta^{(mklr)}_1$ is the reduced simple normal crossing divisor on $X^{(mklr)}_1$ uniquely determined by the adjunction formula  
		\[
		\widetilde{\pi}^\ast\Big(\sum_{i=1}^{mklr} p_i^\ast\big(K_{X_1/S_1} + \Delta_1\big)\Big) = K_{X^{(mklr)}_1/S_1} + \Delta^{(mklr)}_1 - E,  
		\]  
		where $E$ is an effective $\widetilde{\pi}$-exceptional $\mathbb{Q}$-divisor whose support contains no irreducible component of $\Delta^{(mklr)}_1$.  		
		Denote the induced structure morphism by  
		\[
		f^{(mklr)}_1 \colon (X^{(mklr)}_1, \Delta^{(mklr)}_1) \to S_1.  
		\] 
		
		By performing blow-ups of $\widetilde{X}^{(mklr)}_1$ along centers supported over $\widetilde{S} \setminus \widetilde{S}^o$ we may assume the existence of a commutative diagram  
		\begin{align*}  
			\xymatrix{  
				(\widetilde{X}^{(mklr)}_1,\widetilde{\Delta}^{(mklr)}_1)\ar[r]^{\tau^{(mklr)}}\ar[d]^{\tilde{f}^{(mklr)}_1} & (X^{(mklr)}_1,\Delta_1^{(mklr)})\ar[d]^{f^{(mklr)}_1}\\  
				\widetilde{S}_1 \ar[r]^\sigma& S_1  
			},  
		\end{align*}  
		in which $\tau^{(mklr)}$ is a generically finite projective morphism.  
		
		By construction, one has  
		$$\widetilde{\Delta}^{(mklr)}|_{(\sigma \circ \widetilde{f}^{(mklr)})^{-1}(S^o)} = (\tau^{(mklr)})^{-1}(\Delta^{(mklr)})|_{(\sigma \circ \widetilde{f}^{(mklr)})^{-1}(S^o)}.$$  
		Since the pair $(\widetilde{X}^{[mklr]}_{1,\widetilde{S}_1}, \widetilde{\Delta}^{[mklr]}_{1,\widetilde{S}_1})$ has only semi-log canonical singularities, it follows that  
		$$\widetilde{\Delta}^{(mklr)}_1 \leq (\tau^{(mklr)})^{-1}(\Delta^{(mklr)}_1) + (\sigma \circ \widetilde{f}^{(mklr)})^{-1}(D).$$  
		Therefore, by Lemma \ref{lem_Viehweg_inclusion}, there is an inclusion  
		\begin{align}\label{align_pullback_pushforward}  
			&\sigma_\ast \widetilde{f}^{(mklr)}_{1\ast}\left(\mathcal{O}_{\widetilde{X}_1^{(mklr)}}\big(kK_{\widetilde{X}_1^{(mklr)}/\widetilde{S}_1} + k\widetilde{\Delta}_1^{(mklr)}\big)\right) \\  
			\nonumber  
			\subset\ &\sigma_\ast \sigma^\ast f^{(mklr)}_{1\ast}\left(\mathcal{O}_{X_1^{(mklr)}}\big(kK_{X_1^{(mklr)}/S_1} + k\Delta^{(mklr)} + (f^{(mklr)})^{-1}(kD)\big)\right).  
		\end{align}  
		Moreover, Lemma \ref{lem_functorial_pushforward_pluricanonical} yields an injective morphism  
		\begin{align*}  
			\left(f_{1\ast}\mathcal{O}_{X_1}\big(kK_{X_1/S_1} + \lfloor k\Delta - k f^{-1}R_f \rfloor\big)\right)^{\otimes mklr}  
			\hookrightarrow \sigma_\ast\left( \widetilde{f}_{1\ast}\left(\mathcal{O}_{\widetilde{X}_1}\big(kK_{\widetilde{X}_1/\widetilde{S}_1} + k\widetilde{\Delta}\big)\right)^{\otimes mklr}\right).  
		\end{align*}  
		Combining this with Lemma \ref{lem_mild_pushforward} and (\ref{align_pullback_pushforward}), we obtain an injective morphism  
		\begin{align}\label{align_main_proof_inclusion} 
			&\left(f_{1\ast}\mathcal{O}_{X_1}\big(kK_{X_1/S_1} + \lfloor k\Delta - k f^{-1}R_f \rfloor\big)\right)^{\otimes mklr} \\  
			\nonumber  
			\hookrightarrow\ &\sigma_\ast \sigma^\ast f^{(mklr)}_{1\ast}\left(\mathcal{O}_{X_1^{(mklr)}}\big(kK_{X_1^{(mklr)}/S_1} + k\Delta^{(mklr)} + (f^{(mklr)})^{-1}(kD)\big)\right) \\  
			\nonumber  
			\stackrel{\mathrm{tr}}{\longrightarrow}\ &f^{(mklr)}_{1\ast}\left(\mathcal{O}_{X_1^{(mklr)}}\big(kK_{X_1^{(mklr)}/S_1} + k\Delta^{(mklr)} + (f^{(mklr)})^{-1}(kD)\big)\right),  
		\end{align}  
		where $\mathrm{tr}$ denotes the natural trace map associated with the finite morphism $\sigma$. 
		
		\emph{Step 2: Proof of the inequalities.} 
	    Without loss of generality, we may assume that both $W$ and  
	    $$f_\ast\left(\mathcal{O}_{X}\big(kK_{X/S} + \lfloor k\Delta - k f^{-1}R_f \rfloor\big)\right)$$  
	    are locally free sheaves on $S_1$.  
	    The inclusion  
	    $$W|_{S_1} \hookrightarrow f_\ast\left(\mathcal{O}_{X}\big(kK_{X/S} + \lfloor k\Delta - k f^{-1}R_f \rfloor\big)\right)^{\otimes r}|_{S_1}$$  
	    induces an injective morphism  
	    \begin{align}\label{align_detW_inj}  
	    	(\det W)^{\otimes k}|_{S_1} \hookrightarrow \bigotimes^{kl} W|_{S_1} \hookrightarrow f_{1\ast}\left(\mathcal{O}_{X_1}\big(kK_{X_1/S_1} + \lfloor k\Delta - k f^{-1}R_f \rfloor\big)\right)^{\otimes klr}.  
	    \end{align}  
	    In summary, we obtain a chain of morphisms  
	    \begin{align*}  
	    	(\det W)^{\otimes mk}|_{S_1} &\hookrightarrow f_{1\ast}\left(\mathcal{O}_{X_1}\big(kK_{X_1/S_1} + \lfloor k\Delta - k f^{-1}R_f \rfloor\big)\right)^{\otimes mklr} \quad \text{(\ref{align_detW_inj})} \\  
	    	&\hookrightarrow f^{(mklr)}_{1\ast}\left(\mathcal{O}_{X_1^{(mklr)}}\big(kK_{X_1^{(mklr)}/S_1} + k\Delta^{(mklr)} + (f^{(mklr)})^{-1}(kD)\big)\right) \quad \text{(\ref{align_main_proof_inclusion})}.  
	    \end{align*}  
	    The composition is nonzero: its restriction to the open subset $S^o \subset S_1$ coincides with the composition  
	    $$(\det W)^{\otimes mk}|_{S^o} \hookrightarrow f^o_\ast\left(\mathcal{O}_{X^o}\big(kK_{X^o/S^o} + k\Delta^o\big)\right)^{\otimes mklr} \simeq f^{[mklr]}_\ast\left(\mathcal{O}_{X^{o[mklr]}}\big(kK_{X^{o[mklr]}/S^o} + k\Delta^{o[mklr]}\big)\right),$$  
	    which is injective. Consequently, after twisting by $\mathcal{O}_S(-kD)$ and an ideal sheaf $I_Z$ whose cosupport is contained in $Z$, we obtain an injective morphism  
	    \begin{align}  
	    	(\det W)^{\otimes mk} \otimes \mathcal{O}_S(-kD) \otimes I_Z^{\otimes k} \hookrightarrow f^{(mklr)}_\ast\left(\mathcal{O}_{X^{(mklr)}}\big(kK_{X^{(mklr)}/S} + k\Delta^{(mklr)}\big)\right).  
	    \end{align}  
		Applying Theorem \ref{thm_VZ_Arakelov} to the morphism $(X^{(mklr)}, \Delta^{(mklr)}) \to S$, which is a simple normal crossing family over $S \setminus D_f$ of relative dimension $m k l r n$, and to the torsion-free sheaf $(\det W)^{\otimes m} \otimes \mathcal{O}_S(-D) \otimes I_Z$, we obtain the inequality  
		\begin{align}  
			\big(m c_1(W) - 2D\big) \cdot \alpha \leq \frac{1}{m k l r n + 1} \sum_{p=1}^{m k l r n} p \, d^{p-1} \, (K_S + D) \cdot \alpha  
		\end{align}  
		for every movable curve class $\alpha \in N_1(S)$ and every positive integer $m$.  
		
		If, in addition, $K_S + D$ is ample, Theorem \ref{thm_VZ_Arakelov} yields  
		\begin{align}  
			\big(m c_1(W) - 2D\big) \cdot (K_S + D)^{d-1} \leq \frac{m k l r n}{d} \, (K_S + D)^d.  
		\end{align}  
		Rearranging terms gives  
		\begin{align}  
			c_1(W) \cdot (K_S + D)^{d-1} \leq \frac{k l r n}{d} \, (K_S + D)^d + \frac{2}{m} \, D \cdot (K_S + D)^{d-1}.  
		\end{align}  
		Letting $m \to +\infty$, the second term on the right-hand side vanishes, and we conclude  
		\begin{align}  
			c_1(W) \cdot (K_S + D)^{d-1} \leq \frac{k l r n}{d} \, (K_S + D)^d.  
		\end{align}
	\end{proof}
\begin{cor}\label{cor_Arakelov_ineq}
	Let $f \colon (X,\Delta) \to S$ be a morphism and let $R_f \subset D \subset D_f \subset S$ be the associated divisors as in Theorem \ref{thm_Arakelov_family}. Let  
	$$W \subset \left(f_\ast\left(\mathcal{O}_X\big(kK_{X/S} + \lfloor k\Delta \rfloor\big)\right)\right)^{\otimes r}$$  
	be a coherent subsheaf, for some integers $k,r \geq 1$, with $\operatorname{rank}(W) = l$. If $K_S + D$ is pseudo-effective, then for every movable curve class $\alpha \in N_1(S)$ and every positive integer $m$, one has  
	\begin{align}\label{align_Arakelov_cor_1}  
		c_1(W) \cdot \alpha \leq \frac{1}{m(m k l r n + 1)} \sum_{p=1}^{m k l r n} p \, d^{p-1} \, (K_S + D) \cdot \alpha + \frac{2}{m} \, D \cdot \alpha + r \, \lceil k R_f \rceil \cdot \alpha.  
	\end{align}  
	
	Moreover, the following refined estimates hold:  
	\begin{itemize}  
		\item If $\dim S = 1$, then  
		\begin{align}\label{align_Arakelov_cor_2}  
			\deg W \leq \frac{k l r n}{2} \, \deg(K_S + D) + r \, \deg \lceil k R_f \rceil.  
		\end{align}  
		\item If $(S,D_f)$ is log smooth (so that $D = D_f$) and $K_S + D$ is ample, then  
		\begin{align}\label{align_Arakelov_cor_3}  
			c_1(W) \cdot (K_S + D)^{d-1} \leq \frac{k l r n}{d} \, (K_S + D)^d + r \, \lceil k R_f \rceil \cdot (K_S + D)^{d-1}.  
		\end{align}  
	\end{itemize}
\end{cor}
\begin{proof}
	Notice that the condition of $W$ implies that there is an injective morphism
	$$W\otimes\sO_S(-r\lceil kR_f\rceil)\to \left(f_\ast(\sO_X(kK_{X/S}+\lfloor k\Delta-kf^{-1}R_f\rfloor))^{\otimes r}\right)^{\vee\vee}.$$
	Applying Theorem \ref{thm_Arakelov_family} to $W\otimes\sO_S(-r\lceil kR_f\rceil)$ we obtain the inequality (\ref{align_Arakelov_cor_1}) and (\ref{align_Arakelov_cor_3}).
	When $d=\dim S=1$, by taking $m\to +\infty$ in (\ref{align_Arakelov_cor_1}) we obtain (\ref{align_Arakelov_cor_2}).
\end{proof}

    \section{Application: volume inequality}
    Let $X$ be a smooth projective variety and $L \in \operatorname{Pic}(X)$. The Iitaka volume of $L$ is defined by  
    \begin{align}\label{align_Iitaka_volume}  
    	\operatorname{Ivol}(L) := \limsup_{k \to \infty} \frac{\kappa(L)! \, h^0\big(X, L^{\otimes k}\big)}{k^{\kappa(L)}}.  
    \end{align}  
    
    \begin{thm}  
    	Let $f \colon X \to S$ be a morphism from a smooth projective variety $X$ to a smooth projective curve $S$, with geometric generic fiber $F$ of dimension $n$ admitting a good minimal model. Let $D \subset S$ be a reduced effective divisor such that $f$ restricts to a smooth family over $S \setminus D$, and let $R_f \subset D$ denote the ramification divisor of $f$. Let $L$ be a nef line bundle on $S$. If $K_S + D$ is nef, then  
    	\begin{align*}  
    		\operatorname{Ivol}\big(\omega_{X/S} \otimes f^\ast L\big)  
    		\leq \operatorname{Ivol}(\omega_F) \cdot \big(\kappa(\omega_F) + 1\big) \cdot \left( \frac{n}{2} \deg(K_S + D) + \deg L \right)  
    		+ \limsup_{k \to \infty} \frac{\big(\kappa(\omega_F) + 1\big)!}{k^{\kappa(\omega_F) + 1}} \deg \lceil k R_f \rceil.  
    	\end{align*}  
    \end{thm}  
    \begin{proof}
    	We adapt the argument of \cite{LTZ2017}.  
    	The statement holds trivially when $f$ is isotrivial; hence, we assume $f$ is non-isotrivial. Let $g$ denote the genus of $S$, and let $s = \#D$. By \cite[Theorem 1.1]{Kawamata1985}, the direct image sheaf $f_\ast(\omega_{X/S}^{\otimes k})$ is ample for all integers $k \gg 1$. Since the geometric generic fiber $F$ admits a good minimal model, its canonical ring $\bigoplus_{k \geq 0} H^0(F, \omega_F^{\otimes k})$ is finitely generated (see also \cite{HM2010}). Consequently, there exists an integer $k_0 > 0$ such that, for all $k \geq k_0$ and all integers $l > 1$, the multiplication map  
    	$$\operatorname{Sym}^l f_\ast(\omega_{X/S}^{\otimes k}) \to f_\ast(\omega_{X/S}^{\otimes kl})$$  
    	is surjective over a dense Zariski open subset of $S$.  
    	
    	It follows that, for sufficiently large $l > 1$, there exists an ample line bundle $H$ on $S$ with $\deg H > 2g - 1$ and a morphism  
    	$$\bigoplus^{N} H \to f_\ast(\omega_{X/S}^{\otimes kl}),$$  
    	surjective over a dense Zariski open subset of $S$, for some $N \in \mathbb{N}$. This surjectivity implies the vanishing  
    	\begin{align*}  
    		H^1\big(S,\, f_\ast(\omega_{X/S}^{\otimes k}) \otimes L^{\otimes k}\big) = 0  
    	\end{align*}  
    	for all $k \gg 0$. 
    	
    	By the Riemann-Roch theorem for vector bundles on the curve $S$, we have  
    	\begin{align*}  
    		h^0\big(X,\, \omega_{X/S}^{\otimes k} \otimes f^\ast L^{\otimes k}\big)  
    		&= h^0\big(S,\, f_\ast(\omega_{X/S}^{\otimes k}) \otimes L^{\otimes k}\big) \\  
    		&= \deg f_\ast(\omega_{X/S}^{\otimes k}) + \operatorname{rank} f_\ast(\omega_{X/S}^{\otimes k}) \cdot \big(k \deg L + 1 - g\big).  
    	\end{align*}  
    	Applying inequality (\ref{align_Arakelov_cor_2}) to the morphism $f \colon X \to S$ and the subsheaf $W = f_\ast(\omega_{X/S}^{\otimes k})$ yields  
    	\begin{align*}  
    		\deg f_\ast(\omega_{X/S}^{\otimes k})  
    		\leq \frac{k n \cdot \operatorname{rank} f_\ast(\omega_{X/S}^{\otimes k})}{2} \, (2g - 2 + s) + \deg \lceil k R_f \rceil.  
    	\end{align*}  
    	Substituting this bound gives  
    	\begin{align*}  
    		h^0\big(X,\, \omega_{X/S}^{\otimes k} \otimes f^\ast L^{\otimes k}\big)  
    		\leq \operatorname{rank} f_\ast(\omega_{X/S}^{\otimes k}) \cdot \left( \frac{k n}{2} (2g - 2 + s) + k \deg L + 1 - g \right) + \deg \lceil k R_f \rceil.  
    	\end{align*}  
    	By \cite[Theorem 1.1]{Kawamata1985} and \cite[Lemma 2.3.31]{Fujino2020}, the Kodaira dimension satisfies  
    	\begin{align*}  
    		\kappa\big(\omega_{X/S} \otimes f^\ast L\big) = \kappa(\omega_F) + 1.  
    	\end{align*}  
    	Hence, the Iitaka volume is bounded as follows:  
    	\begin{align*}  
    		\operatorname{Ivol}\big(\omega_{X/S} \otimes f^\ast L\big)  
    		&= \limsup_{k \to \infty} \frac{\kappa\big(\omega_{X/S} \otimes f^\ast L\big)! \cdot h^0\big(X,\, \omega_{X/S}^{\otimes k} \otimes f^\ast L^{\otimes k}\big)}{k^{\kappa(\omega_{X/S} \otimes f^\ast L)}} \\  
    		&\leq \limsup_{k \to \infty} \frac{(\kappa(\omega_F) + 1)!}{k^{\kappa(\omega_F) + 1}} \Bigg[ h^0(F,\, \omega_F^{\otimes k}) \cdot \bigg( \frac{k n}{2} (2g - 2 + s) + k \deg L + 1 - g \bigg) + \deg \lceil k R_f \rceil \Bigg] \\  
    		&= \operatorname{Ivol}(\omega_F) \cdot (\kappa(\omega_F) + 1) \cdot \bigg( \frac{n}{2} (2g - 2 + s) + \deg L \bigg) \\  
    		&\quad + \limsup_{k \to \infty} \frac{(\kappa(\omega_F) + 1)!}{k^{\kappa(\omega_F) + 1}} \Big( (1 - g) \cdot h^0(F,\, \omega_F^{\otimes k}) + \deg \lceil k R_f \rceil \Big) \\  
    		&= \operatorname{Ivol}(\omega_F) \cdot (\kappa(\omega_F) + 1) \cdot \bigg( \frac{n}{2} (2g - 2 + s) + \deg L \bigg)  
    		+ \limsup_{k \to \infty} \frac{(\kappa(\omega_F) + 1)!}{k^{\kappa(\omega_F) + 1}} \deg \lceil k R_f \rceil.  
    	\end{align*}  
    \end{proof}
	\bibliographystyle{plain}
	\bibliography{SBC}
	
\end{document}